\documentclass[12pt]{amsart}
\usepackage{latexsym,fancyhdr,amssymb,color,amsmath,amsthm,graphicx,listings,comment}
\usepackage[section]{placeins} 
\newtheorem{thm}{Theorem} \newtheorem{lemma}{Lemma}  
\let\paragraph\subsection
  
\newcommand{\RR}{\mathbb{R}}

\title{On Atiyah-Singer and Atiyah-Bott for finite abstract simplicial complexes}
\author{Oliver Knill}
\date{Aug 20, 2017}
\address{Department of Mathematics \\ Harvard University \\ Cambridge, MA, 02138 }
\subjclass{58J20, 68R05, 05E45}
\keywords{Discrete differential geometry, simplicial complexes, graphs, Atiyah-Singer, Atiyah-Bott, 
connection calculus }

\begin{document}
\maketitle

\begin{abstract}
A linear or multi-linear valuation on a finite abstract simplicial complex can be expressed as an 
analytic index ${\rm dim}({\rm ker}(D))$ $-{\rm dim}({\rm ker}(D^*))$ of a differential complex $D:E \to F$. 
In the discrete, a complex $D$ can be called elliptic if a McKean-Singer spectral symmetry applies
as this implies ${\rm str}(e^{-t D^2})$ is $t$-independent. In that case, the analytic index of $D$ is 
$\chi(G,D)=\sum_k (-1)^k b_k(D)$, where $b_k$ is the $k$'th Betti number, which by Hodge
is the nullity of the $(k+1)$'th block of the Hodge operator $L=D^2$. It can also be 
written as a topological index $\sum_{v \in V} K(v)$, where $V$ is the set 
of zero-dimensional simplices in $G$ and where $K$ is an Euler type curvature defined by $G$ and $D$. 
This can be interpreted as a Atiyah-Singer type correspondence between analytic and topological index.
Examples are the de Rham differential complex for the Euler 
characteristic $\chi(G)$ or the connection differential complex for Wu characteristic $\omega_k(G)$. 
Given an endomorphism $T$ of an elliptic complex, the Lefschetz number $\chi(T,G,D)$
is defined as the super trace of $T$ acting on cohomology defined by $D$ and $G$. 
It is equal to the sum $\sum_{v \in V} i(v)$, where $V$ is the set of zero-dimensional simplices
which are contained in fixed simplices of $T$, and $i$ is a Brouwer type index.
This Atiyah-Bott result generalizes the Brouwer-Lefschetz fixed point theorem for an endomorphism
of the simplicial complex $G$. In both the static and dynamic setting, the proof is done by heat 
deforming the Koopman operator $U(T)$ to get the cohomological picture 
${\rm str}(e^{-t D^2} U(T))$ in the limit $t \to \infty$
and then use Hodge, and then by applying a discrete gradient flow to the simplex data defining the valuation to push
${\rm str}(U(T))$ to the zero dimensional set $V$, getting curvature $K(v)$ or the Brouwer type index $i(v)$. 
\end{abstract}

\section{Simplicial complexes}

\paragraph{}
A {\bf finite abstract simplicial complex} is a finite set $G$ of
non-empty sets which is closed under the operation of taking finite non-empty 
subsets. A set $x \in G$ with $k+1$ elements is called a {\bf $k$-simplex} or $k$ dimensional
face; its {\bf dimension} is $k$. If $H \subset G$, we say $H$ is a {\bf sub-complex}. 
The set of sub-complexes of $G$ is a {\bf Boolean lattice} because both the 
union and intersection of a complex is a complex and the empty complex is a complex. 
We often just write {\bf simplicial complex} for a finite abstract simplicial complex. 

\paragraph{}
An example of a simplicial complex is the {\bf Whitney complex} of a finite simple
graph $(V,E)$. In that case, $G = \{ A \subset V \; | \;  \forall a,b \in A, (a,b) \in E \}$. 
The {\bf Barycentric refinement} $G_1$ of a complex $G$
is the subset $G_1$ of of the power set $2^G$ consisting of sets of sets in $G$, in which any pair
$a,b$ either satisfies $a \subset b$ or $b \subset a$. It is the Whitney complex of the
graph $(V,E)=(G,\{(a,b) \: | \; a \subset b \;{\rm or} \; b \subset a \}$). In topology,
one can therefore mostly focus on Whitney complexes of graphs which are more intuitive
than sets of sets. The Barycentric argument shows that almost nothing is lost by looking
at Whitney complexes of graphs.

\paragraph{}
Not all simplicial complexes are Whitney complexes. We can for example truncate a
given complex at dimension $d$, removing all sets of cardinality larger than $d+1$
to get the {\bf $d$-dimensional skeleton complex} of $G$. For a Whitney complex of a
graph, such a skeleton is no more a Whitney complex in general. Take $G=K_3$ for example
which is $G=\{ (1,2,3), (1,2),(2,3),(1,3),(1),(2),(3) \}$. The $1$-dimensional skeleton
is the subcomplex $H=\{ (1,2),(2,3),(1,3),(1),(2),(3) \}$ which is no more the Whitney 
complex of a graph. The subcomplex $H$ is a discrete circle with Euler characteristic $0$
while the complex $G$ itself is a two dimensional disc with Euler characteristic $1$. 
Complexes can appear in other ways also. The {\bf graphical complex}
of a graph consists of all non-empty forests in $G$, subgraphs for which 
every connected component is a tree.  As any non-empty subset of a forest is a forest,
this is a simplicial complex. More general {\bf graph complexes}, where the sets
are subsets of the edge set are considered in \cite{JonssonSimplicial}, which is also
a good introduction for abstract simplicial complexes. 

\section{Valuations}

\paragraph{}
Assume we are given a simplicial complex $G$. 
An integer-valued function $X$ from the Boolean lattice of sub-complexes to the integers 
is called a {\bf valuation} if $X(A \cup B) + X(A \cap B) =X(A) + X(B)$. Examples of 
valuations are $v_k(H)$ counting the number of $k$-simplices in $H$. We don't really insist
in general to have $X$ integer valued. One could assume $X$ to take values in an Abelian group.
According to the {\bf discrete Hadwiger theorem} \cite{KlainRota}, the
linear space of valuations of a complex is $(d+1)$-dimensional, where $d$ is the maximal 
dimension of the complex. A basis is $\{ v_k(G) \}_{k=0}^d$.
If $f(A)=(v_0(A),\dots, v_d(A))$ is the {\bf $f$-vector} of $A$, then 
$X(A) = \langle X,f(A) \rangle = \sum_{i=0}^d X_i v_i(A)$. 

\paragraph{}
An example of a valuation is the {\bf Euler characteristic} $\chi(G) = \sum_{x \in G} \omega(x)$ with 
$\omega(x)=(-1)^{{\rm dim}(x)}$ for a simplex $x$. It is invariant under Barycentric refinements 
and comes from the only eigenvector of eigenvalue $1$ of the transpose $S^T$ of the 
{\bf Barycentric refinement operator} $S_{ij}=i! {\rm Stirling}_2(j,i)$ mapping the $f$-vector of $G$ 
to the $f$-vector of its Barycentric refinement $G_1$. Other examples of valuations
are the number of $0$-dimensional points in $G$ or the {\bf volume}, the number of {\bf facets}, 
sets in $G$ of largest cardinality $d+1$ if $G$ has dimension $d$. Also
the other eigenvectors of the Barycentric refinement operator can be of topological interest.
For geometric graphs, discrete manifolds, in particular, half of the eigenvectors
lead to valuations which are zero. They are related to {\bf Dehn-Sommerville} invariants. 
On the other hand, a {\bf Betti number} $A \to b_i(A)$ is no valuation in general. 

\paragraph{}
A {\bf multi-linear valuation} $X$ is function of $G^k = G \times G \times \cdots \times G$
which is a valuation in each of the $k$ coordinates \cite{valuation}. Examples of {\bf bilinear valuations} are
$v_{kl}(G)$, counting the number of ordered pairs of simplices $(a,b)$ of dimension $k$ and $l$ which intersect.
By a generalization of the discrete Hadwiger theorem, if $G$ has dimension $d$, then
the valuations $v_{k,l}(G)$ with $0 \leq k \leq l \leq d+1$ form a basis for the linear space
of bilinear valuations. If $v_{kl}(G)$ is the symmetric {\bf $f$-matrix} encoding the intersection data, then every 
bilinear valuation can be written $X(A,B) = \sum_{i,j} X_{ij} v_{ij}(G)$, where $X_{ij}$
is a symmetric $(d+1) \times (d+1)$ matrix. Similar statements hold for any $k$-linear valuation.  

\paragraph{}
An example of a bilinear valuation is the {\bf Wu intersection number}  \cite{valuation}
$\omega(A,B) = \sum_{x \sim y} \omega(x) \omega(y)$, 
where the sum is over all ordered pairs of elements $x \in A, y \in B$ 
which intersect. The {\bf Wu characteristic} is then $\omega(G)=\omega(G,G)$. 
Wu characteristic has many important properties: it is invariant under Barycentric refinements, if $x$ is a simplex, 
then $\omega(x)=(-1)^{{\rm dim}(x)}$, and if $G$ is a discrete manifold with $(d-1)$ dimensional boundary, 
then $\omega(G) = \chi(G)-\chi(\delta G)$. 
Also higher order {\bf Wu characteristics} $\omega_k$ exist. The first one, $\omega_1$, is the Euler characteristic $\chi$, 
the second $\omega_2$ is the bilinear Wu characteristic. Each of these characteristic has its own 
differential complex and its own cohomology. This is useful, as now not only the valuations, but also the
Betti numbers are combinatorial invariants. Unlike simplicial cohomology which is invariant under homotopy,
the finer connection cohomologies related to $\omega_k$ are not.

\section{Calculus} 

\paragraph{}
Given a finite abstract simplicial complex $G$, equipped with an orientation of the simplices, one can look at
the linear space $\Lambda_k(G)$ of anti-symmetric functions on the set of $k$-simplices. The initial
choice of the orientation of $G$ is a basis selection. It is a gauge choice as custom in linear algebra which 
does not affect interesting quantities. 
The orientations on each simplex $x \in G$ do not have to be compatible so that every abstract simplicial complex
$G$ can be oriented. A complex can be called {\bf compatibly orientable}, if there is a choice of orientation which is 
compatible: if $x \subset y$ are simplices in $G$, then the orientation of $x$ is inherited from the orientation of $y$.
As in the continuum, examples like the M\"obius strip or Klein bottle show that not all complexes possess a 
{\bf compatible orientation}. But every simplicial complex can be oriented similarly as 
any vector space can be equipped with a basis. 

\paragraph{}
The space $\Lambda_k(G)$ is also called the space of {\bf discrete $k$-forms} on $G$. It has 
dimension $v_k(G)$. We think of each $\Lambda_k(G)$ as a fiber bundle over $G$. We can extend a 
form $f$ from sets in $G$ to sub-complexes of $G$ to get {\bf signed valuations} 
$f(A) = \sum_{x \in A} f(x)$ which still satisfy $f(A \cap B)  + f(A \cap B) = f(A) + f(B)$. 
Evaluating a signed valuation is {\bf integration} $f(A) = \int_A f$. 

\paragraph{}
The integration of signed valuations corresponds to {\bf geometric measure theory}, while the integration 
of valuations corresponds to {\bf geometric probability theory}={\bf integral geometry}. 
The former is orientation sensitive 
like line or flux integrals in school calculus. The later does not depend on the orientation 
and relates to integrations in calculus, like arc length or surface area. 
In the discrete, it can be important to be aware of the difference and distinguish 
integration of valuations and integration of signed valuations.

\paragraph{}
The {\bf exterior derivatives} $d_k: \Lambda_k(G) \to \Lambda_{k+1}(G)$ are linear 
transformations which extend to a linear transformation $d$ on the graded vector
space $\Lambda= \oplus \Lambda_k$, the vector space of {\bf discrete differential forms}. 
A differential form is just a function from $G$ to $\RR$ satisfying $f(T(x)) = {\rm sign}(T) f(x)$
for any permutation $T$ of the simplex $x$. The boundary operation $\delta$ maps
a sub complex $A$ to its {\bf boundary chain} $\delta A$, where the orientation of $\delta A$ 
is now compatible with the orientation of $A$. The image $\delta A$ of a complex
is not a complex any more in general. For $A=xu+yu+zu+x+y+z+u$ for example, then 
$\delta A = (u-x) + (u-y) + (u-z)=3u-x-y-z$. 
A signed valuation $f$ can be extended linearly to the group of chains on $G$ however. 

\paragraph{}
The exterior derivative $d$ defines what we will call a 
discrete {\bf elliptic complex} $D=(d+d^*):E \to F$: 
the non-zero eigenvalues of the Hodge operator $D^2$ restricted to the
space $E$ of even forms, are in bijective correspondence with the eigenvalues of $D^2$ 
restricted to the space $F$ of odd forms. When talking about an elliptic complex, we have three
things: a linear map $D=d+d^*$ incorporating all exterior derivatives $d_k$, where $d^2=0$, 
a domain $E$ and a co-domain $F$. The exterior derivatives $d_k$ can be rather general however. 
   
\paragraph{}
If $f$ is a signed valuation, {\bf Stokes theorem} is $f(\delta A)) = df(A)$.
For example, if $A$ is a two-dimensional connected sub-complex of $G$ equipped with a compatible
orientation and having the property that every 
unit sphere $S(x)$ in $A$ is either a circular graph $C_n$ with $n \geq 4$ or a linear
graph $L_n$ with $n \geq 2$, then $A$ is called a {\bf surface}. The boundary $\delta A$ is then
a {\bf curve}, a one-dimensional complex consisting of finitely many circular directed graphs.
Stokes theorem is then the classical Stokes theorem from 
school calculus. If $A$ is a sub-complex for which every unit sphere is either a $0$-sphere
(a 2-point graph without edges), or a 1-point graph $P_1$, then every connected component
has either $0$ or $2$ boundary points and $\int_A df = \int_{\delta A} f$ is the {\bf fundamental 
theorem of line integrals}. If $A$ is a sub-complex for which every unit sphere is a $2$-sphere
or a 2-disc, then the boundary $\delta A$ is a discrete closed 2-manifold, a complex for
which every unit sphere is a circular graph. In that case, Stokes theorem corresponds to 
the {\bf classical divergence theorem}.

\paragraph{}
Every differential complex $D=d+d^*$ defines a {\bf flavor of calculus}. In each case,
Stokes theorem is the defining relation for the {\bf boundary operation} $\delta$. 
The boundary $\delta x$ of a simplex $x$ is now no more given by a collection of simplices or chain.
It must be probed with functions $f(\delta x) = df(x)$. The boundary $\delta x$ can now be 
more extended unlike in classical or in {\bf connection calculus}, where in the quadratic case, 
the boundary $\delta$ acts on pairs of intersecting simplices as 
$\delta (x,y) = (\delta x,y) - (x,\delta y)$. 

\paragraph{}
A example where the boundary $\delta$ is no more in sync with the geometric boundary 
is if $D(t)=d_t+d^*_t+b_t$ is an isospectral Lax deformation $D'=[B,D]$
with $B=d-d^*$ of the Dirac operator $d+d^*$.  \cite{IsospectralDirac,IsospectralDirac2}
This is a deformation $d_t+d_t^*$ of a complex. 
The boundary $\delta_t A$ of a geometric object $A$ is now not a linear 
combination of simplices any more. 
Stokes theorem $f(\delta A)) = df(A)$ is still a definition. But now, 
a traditional line integral of $df$ along a closed loop is no more zero in general
as a closed loop can not be written as a linear combination of boundaries
of simplices any more. In the deformed calculus, fields have appeared. We see that if we let a
geometric object evolve freely in its isospectral set, then the geometry dynamically produces
sources for each space of differential forms. This works in the same way for
Riemannian manifolds, where the deformed exterior derivative $d(t)$ is a pseudo differential 
operator describing an expanding space.

\paragraph{}
One can speculate that the isospectral deformation of the
differential complex produces fields which are relevant in physics. Start with the classical
point of view of Gauss writing a force field $f$ as a {\bf Poisson equation} ${\rm div}(f) = -4\pi G \sigma$, 
where $\sigma$ is the mass density. If the field $f$ is a $2$-form satisfying $df=0$,
then $0=f(\delta A)$ for the {\bf deformed boundary} of a region $G$. For a {\bf classical boundary} $\delta_0 A$ 
of a ball $A$ however, the flux $f(\delta_0 A)$ is not necessarily zero. In the deformed differential complex, it 
appears as if some field $f$ has been generated from geometry alone. It would be good to explore
how the strength of this field depends on the geometry and especially on the curvature. 

\section{Differential complexes} 

\paragraph{}
A {\bf discrete differential complex} is defined as a sequence 
of linear maps $d: E_k \to E_{k+1}$ with $d_{k+1} d_k=0$ with
$E=\bigcup_k E_{2k}$ and $F=\bigcup_k E_{2k+1}$. To be more concrete, 
the finite dimensional vector spaces $E_k$ are required to be subspaces of tensor products
of the de-Rham complex $\Lambda_k$ or connection complex. We ask this so that the individual
fibers are local. A complex defines a {\bf Dirac operator} $D=\sum_i d_i + d_i^*$
with domain $E$ and co-domain $F$. In the case of the de Rham complex $\Lambda_k$, 
we can take $E$ the set of even forms and $F$ the set of odd forms. 

\paragraph{}
Given a differential complex, the analytic index of the Dirac operator $D:E \to F$.
is defined as ${\rm dim}({\rm ker}(D)) - {\rm dim}({\rm ker}(D^*))
={\rm dim}(E)-{\rm dim}(F)$. For example, for the connection complex of order $n$, for which $\Lambda_k(G)$ 
has as dimension the number of $n$-tuples of simultaneously intersecting simplices adding to 
dimension $k$, then the analytic index of $D$ is the $n$'th Wu characteristic. 
If $E=F=\Lambda(G)$, then the analytic index is zero as the kernel of $D$ and $D^*$
agree. 

\paragraph{}
{\bf Examples.} \\
1) Let $E$ be the linear space of even differential forms
and $F$ the linear space of odd differential forms.
The analytic index of $D=(d+d^*)$ is the Euler characteristic $\chi(G)$.  \\
2) Let $E=\Lambda_k(G)$ and $F=\{0\}$. Then, the analytic index of $D=d_k$ is $v_k$. \\
3) Let $E=\Lambda=\oplus_k \Lambda_{2k}$ and $F=\{0\}$ and $D=\sum_i a_{2i} d_{2i}$ 
with $a_i \neq 0$, $d_{i}=0$ for all $i$, then the analytic index of $D$ is $\sum_i v_{2i}$.
This is clearly not an elliptic complex.  \\
4) The analytic index $d_0+d_1$ mapping even forms $E=\oplus \Lambda_{2k}$
to odd forms $F=\oplus \Lambda_{2k+1}$ is $b_0-b_1$, if $G$ is equipped with the
1-skeleton simplicial complex. For $G=K_3$ with this complex $\chi(G)=0$ as the
two dimensional simplex does not count, the complex is a discrete circle. \\

\paragraph{}
If $A$ is a sub simplicial complex of $G$, we have a {\bf sub differential complex}
$D|A: E|A \to F|A$. 
Any subcomplex $A$ of $G$ with the same base $V=\bigcup_{x \in G} x$ has 
its own Euler characteristic $\chi(A)= \sum_{x \in A} \omega(x)$ which can
be written as $\sum_i (-1)^k b_i(A)$, where $b_i$ are the Betti numbers of the 
sub differential complex. 

\begin{lemma}
For any discrete differential operator $D:E \to F$, 
the map which assigns to a sub-complex $A$ the analytic index of $D|A$ is a valuation. 
Every linear or multi-linear valuation can be represented as an analytic index
of some differential complex. 
\end{lemma}
\begin{proof}
The kernels $A \to {\rm ker}(d_i|A)$ of $d_i$ are valuations. So are the kernels of $d_i^*$ 
and the sums. To get a linear valuation $v_k(A)$ counting the number of $k$-dimensional
simplices in $A$, take $E_k=\Lambda_k(G)$ and all other $E_j=0$. Then let all $d_j=0$. 
Now, ${\rm dim} {\rm ker}(d_k|A) - {\rm dim} {\rm ker}(d_k^*|A) = {\rm dim}(E_k)=v_k(A)$. \\
To get $v_{ij}$ for example, let all $E_k$ be zero except $E_{i+j}$ 
which is the vector space of all functions on 
$(x,y)$ with ${\rm dim}(x)=i, {\rm dim}(y)=j$. Let all $F_k=\{0\}$. Let all $d_k=0$. 
Now, ${\rm ker}(D)-{\rm ker}(D^*)=v_{ij}$. 
\end{proof} 

\paragraph{}
The reason to focus on Fredholm indices rather then the nullity of the operator itself 
is that they have a chance of staying bounded in continuum limits
and also because $i(A B) = i(A) + i(B)$. In finite dimensions, the Fredholm indices 
is just $i(A)={\rm dim}(E)-{\rm dim}(F)$, independent of $A$. This follows from the 
{\bf rank-nullity theorem} and the fact that the row and column ranks of a 
finite matrix $A$ are the same.

\paragraph{}
In the discrete, Atiyah-Singer or Atiyah-Bott like results still have
some interest as we can equate both with cohomological data as well as topological data
with the valuation, at least  if the complex is elliptic. Classically, ellipticity is defined by
the {\bf symbols} of the differential operators. Instead of trying to translate
a continuum definition to the discrete, we have chosen to define 
ellipticity in the simplest way to have the proofs work. 

\paragraph{}
A {\bf differential complex} $(D,E,F)$ defined by maps 
$\Lambda_k \to^{d_k} \Lambda_{k+1}$ is called an
{\bf elliptic complex} if $L=(d+d^*)^2$ has the property that the spectrum of 
non-zero eigenvalues of $L$ on $E$ is the same than the spectrum
of non-zero eigenvalues of $L$ on the odd forms $F$. 
We wrote "spectrum" rather than "set" to stress that also the multiplicities of the 
eigenvalues have to be the same. The simplest proof of McKean-Singer \cite{McKeanSinger} 
relies on this symmetry \cite{Cycon} and can be adapted to the discrete
\cite{knillmckeansinger}. 

\section{Theorems}

\paragraph{}
The discrete version stated here only requires knowledge of finite sets
and finite matrices. It is a first attempt to emulate those classical 
theorems, risking of course to appear preposterous.

\paragraph{}
Given an elliptic differential complex $D:E \to F$ over a simplicial complex $G$. 
The {\bf analytic index} of $D$ is ${\rm dim}({\rm ker}(D))-{\rm dim}({\rm ker}(D^*))$.
The {\bf cohomological index} of $D$ is $\sum_i (-1)^k b_k(G,D,E,F)$, where $b_k$ are the
Betti numbers defined by the cohomology of $D$. The curvature of a pairwise intersecting
simplex tuple $x=(x_1, \dots, x_k)$ is $\prod_j \omega(x_j)$. The curvature of a vertex $v \in V$
is defined as $K(v)=(\sum_{x \in G^k(v)} i(x))/\sum_{x \in G^k(v)} 1)$, 
where both sum is over the set $G^k(v)$ of pairwise intersecting $k$ tuples 
$(x_1,\dots,x_k), v \in \bigcup x_j$. 
The {\bf topological index} is then defined as $\sum_{v \in V(G)} K(v)$.
In the Gauss-Bonnet case, $K(v) = 1 - \frac{V_0}{2} + \frac{V_1}{3} - \frac{V_2}{4} + \cdots$,
where $V_k(v)$ is the number of $k$-simplices in the {\bf unit sphere} $S(x)$ (often called link in the
simplicial complex literature). 
This formula \cite{cherngaussbonnet} appeared already in \cite{Levitt1992}. Almost a hundred
year old is the planar case where the curvature is $1-d(v)/6$ with vertex degree $d(v)$ appeared
in the context of graph colorings. 

\begin{thm}[Atiyah-Singer like]
The analytic index of $D$ is equal to the cohomological index and equal to the topological index.
\end{thm}
\begin{proof}
The super trace ${\rm str}(e^{-t D^2})$ is independent of $t$. For $t=0$
it is the super trace of $1$ which is the analytic index of $D$. 
Now apply the heat kernel deformation to make the Euler-Poincar\'e equivalence
to cohomological data. For $t \to \infty$, the non-zero eigenspace of $D^2$ dies
out and only the kernels survive. Since by Hodge, the dimensions of the kernels of $L$
restricted to $k$-forms is $b_k(G)$, the super trace in the limit $t \to \infty$ is
the cohomological index. \\
The topological index is obtained by pushing
the defining combinatorial data located on simplices to the zero dimensional part of space.
This can be done in various ways. The above definition of $K(v)$ does this by distributing
each value equally to its vertices \cite{cherngaussbonnet}.
An other extreme case is Poicar\'e-Hopf \cite{poincarehopf} which lets the curvature flow along 
the gradient of a function $f$ having the effect that curvature remains as integer index.
\end{proof}

\paragraph{}
Atiyah-Bott is a Lefschetz type result which relates a cohomologically defined Lefschetz
number with a sum of indices of fixed points of the endomorphism of an elliptic complex.
The proof in \cite{brouwergraph} for Whitney complexes works for general simplicial complexes. 
Simpler is a heat deformation approach. 

\paragraph{}
Let $T$ be an automorphism of the elliptic complex. The {\bf Lefschetz number} of $T$
is the super trace of the linear map induced on cohomology. Let ${\rm Fix}(T)$ be the 
set of elements in $G$ which are fixed under $T$ and let $V(T)$ denote the set
of vertices which contain an element in ${\rm Fix}(T)$. Define for $x \in {\rm Fix}(T)$ the
index $i(x)$ as $(-1)^{{\rm dim}(x)} {\rm det}(T|x) {\rm Tr}(D|x)$. The index
of $v \in V$ is then defined as $i(v) = \sum_{x \in G, v \in x} i(x)/({\rm dim}(x)+1)$. 

\begin{thm}[Atiyah-Bott like]
The Lefschetz number of $T$ is equal to the index sum over all fixed points.
\end{thm} 
\begin{proof}
Also here, there are two deformations: the first Euler-Poincar\'e deformation equates a
sum of fixed point indices with the Lefschetz number, the second,
the Gauss-Bonnet or Poincar\'e-Hopf deformation expresses this Lefschetz number
as an integral of curvature over space, where in the Poincar\'e-Hopf case, the curvature is a
divisor. Averaging the curvature over all locally injective functions gives the Euler curvature.
\end{proof}

\paragraph{}
In the elliptic case, the cohomological data
do not change when making a topological deformation like a Barycentric refinement.
The curvature data however change. This can be exploited for fixed point
results. The Atiyah-Singer or Atiyah-Bott theorems allow for more flexibility
as one can chose also the elliptic complex. The de Rham complex or the 
more general connection complexes are the guiding examples. 

\section{Remarks}

\paragraph{}
Both Atiyah-Singer and Atiyah-Bott are milestones in geometry
which require a decent amount of technical background in functional analysis,
differential geometry and topology \cite{PalaisSeminar,Shanahan}. Heat approaches have
been established in the continuum \cite{BerlineGetzlerVergne}, first by V.K. Patodi. 
The above results have much less structure
as they are defined for general abstract finite simplicial complexes and
don't assume that the geometries have any manifold type. 
The analytic index ${\rm dim} {\rm ker}(D)- {\rm dim} {\rm ker}(D^T)$ for a finite dimensional 
linear operator $D: E \to F$ is always just ${\rm dim}(E)-{\rm dim}(F)$ and so independent of $D$. 
So, if we look at a discrete analogue of an elliptic complex, then the analytic index is already 
the combinatorial quantity under consideration. 

\paragraph{}
Whether there is in the manifold case a refinement-
averaging procedure which produces the classical results, is not clear at the moment.
Especially the topological index has some flexibility still. The curvature should depend
naturally on the elliptic complex. If $D = R D_0 R^*$, where $D_0$ is the connection Dirac operator, then
the curvature can be obtained as the expectation of the Poincar\'e-Hopf indices $E[i_f]$
where the measure on the functions is the push forward of the uniform measure by $R$.
So at least in the case of a Hamiltonian deformation of the complex there is a natural way
to deform the curvature by deforming the measure on the space of functions and getting
curvature through a deformed expectation. 

\paragraph{}
The mathematics underlying the geometry of finite sets is about a century older
than Atiyah-Singer. Combinatorial versions hardly replace the continuum results
but this note could one day lead to a {\bf pedagogical entry point} into the topic. 
I personally still struggle to understand the Atiyah-Singer and the Atiyah-Bott index theorems. 
The theorems \footnote{We mean of course a treatment with full proofs.}
have not yet entered undergraduate courses. It will probably need an other half of a century to
achieve that. And this is important: to cite Atiyah from \cite{RS}: 
{\it "The passing of mathematics on to subsequent generations is
essential for the future, and this is only possible if every generation of mathematicians
understands what they are doing and distills it out in such a form that it is easily
understood by the next generation."}

\paragraph{}
The connection of discrete with the continuum needs still to be 
explored. Maybe the discrete case can in a limiting case become 
the continuum case. It is also possible that the discrete case
remains a {\bf caricature} \footnote{There is an SNL
sketch from 1992 with Tom Hanks of an ``prize is right" show, where a {\bf cordless telephone}
is sold to contestants. The deal was a traditional phone, from which the cord has been cut.}.
We think however that taking Barycentric limits combined with a suitable smoothing process
can lead to classical differential operators which are Fredholm 
and have a finite index. But there are other battles which need to be 
fought first in the discrete.

\paragraph{}
An other link between the discrete and continuum is integral geometry:
if $\gamma: [a,b] \to M, t \to r(t)$ is a smooth curve in a smooth connected manifold $M$,
let $L(\gamma)=\int_a^b |r'(t)| \; dt$ denote its length. Given a probability space
$(\Omega,P)$ of smooth functions $\omega$ on $M$ one can look at the random variable
$N_{\gamma}(\omega)$ counting the number of intersections
of surfaces $\{ \omega=0 \}$ with $\gamma$. Counting the number $N_{\gamma}$ of transitions
from $\omega \leq 0$ to $\omega>0$ defines a {\bf Crofton pseudo metric}
$d(x,y) = \inf_{\gamma(x,y), N_\gamma \in L^1(\Omega,P)} {\rm E}[N_{\gamma}]$,
where the infimum is taken over all curves connecting $x$ with $y$ with the understanding that
$d(x,y) = \infty$, if there should be no $\gamma$ for which $N_\gamma$ is in $L^1$.
The {\bf Kolmogorov quotient} $(M_P,d_P)$ consists of all equivalence classes of the 
equivalence relation given by $d(x,y)=0$. For discrete measures $P$, 
one gets like this discrete metric spaces and so finite simplicial complexes. 
Nash embedding $M$ into an ambient Euclidean space and taking a rotationally invariant
measure $P$ leads to the Riemannian metric on $M$ because it is the Eucliden metric
in the ambient space. As curvature in the discrete can be 
expressed integral geometrically \cite{indexexpectation,colorcurvature}, 
also Gauss-Bonnet type results should go over.  
Bridging the functional analysis of the Dirac and Laplace operators 
and the topological index both remain complicated tasks and as the continuum is 
technical, no real short cut might exist. 

\paragraph{}
Elliptic differential complexes can be added and multiplied and so
extended to a ring over a fixed simplicial complex $G$. As 
we can also look at the {\bf strong ring} generated by simplicial complexes,
there is an other possibility: extend the strong ring of simplicial 
complexes to a strong ring of differential complexes \cite{StrongRing}. 
Now it appears that not only the category of differential complexes 
over simplicial complexes but also the sub-category of elliptic 
differential complexes is a cartesian closed category. 

\vfill

\pagebreak

\bibliographystyle{plain}

\begin{thebibliography}{10}

\bibitem{Cycon}
H.L. Cycon, R.G.Froese, W.Kirsch, and B.Simon.
\newblock {\em {Schr\"odinger} Operators---with Application to Quantum
  Mechanics and Global Geometry}.
\newblock Springer-Verlag, 1987.

\bibitem{JonssonSimplicial}
J.~Jonsson.
\newblock {\em Simplicial Complexes of Graphs}, volume 1928 of {\em Lecture
  Notes in Mathematics}.
\newblock Springer, 2008.

\bibitem{KlainRota}
D.A. Klain and G-C. Rota.
\newblock {\em Introduction to geometric probability}.
\newblock Lezioni Lincee. Accademia nazionale dei lincei, 1997.

\bibitem{cherngaussbonnet}
O.~Knill.
\newblock A graph theoretical {Gauss-Bonnet-Chern} theorem.
\newblock {\\}http://arxiv.org/abs/1111.5395, 2011.

\bibitem{poincarehopf}
O.~Knill.
\newblock A graph theoretical {Poincar\'e-Hopf} theorem.
\newblock {\\} http://arxiv.org/abs/1201.1162, 2012.

\bibitem{indexexpectation}
O.~Knill.
\newblock On index expectation and curvature for networks.
\newblock {\\}http://arxiv.org/abs/1202.4514, 2012.

\bibitem{knillmckeansinger}
O.~Knill.
\newblock {The McKean-Singer Formula in Graph Theory}.
\newblock {\\}http://arxiv.org/abs/1301.1408, 2012.

\bibitem{brouwergraph}
O.~Knill.
\newblock A {Brouwer} fixed point theorem for graph endomorphisms.
\newblock {\em Fixed Point Theory and Applications}, 85, 2013.

\bibitem{IsospectralDirac2}
O.~Knill.
\newblock An integrable evolution equation in geometry.
\newblock {{\\} http://arxiv.org/abs/1306.0060}, 2013.

\bibitem{IsospectralDirac}
O.~Knill.
\newblock Isospectral deformations of the {D}irac operator.
\newblock {{\\}http://arxiv.org/abs/1306.5597}, 2013.

\bibitem{colorcurvature}
O.~Knill.
\newblock Curvature from graph colorings.
\newblock {{\\}http://arxiv.org/abs/1410.1217}, 2014.

\bibitem{valuation}
O.~Knill.
\newblock Gauss-{B}onnet for multi-linear valuations.
\newblock {\\}http://arxiv.org/abs/1601.04533, 2016.

\bibitem{StrongRing}
O.~Knill.
\newblock The strong ring of simplicial complexes.
\newblock {\\}https://arxiv.org/abs/1708.01778, 2017.

\bibitem{Levitt1992}
N.~Levitt.
\newblock The {E}uler characteristic is the unique locally determined numerical
  homotopy invariant of finite complexes.
\newblock {\em Discrete Comput. Geom.}, 7:59--67, 1992.

\bibitem{McKeanSinger}
H.P. McKean and I.M. Singer.
\newblock Curvature and the eigenvalues of the {L}aplacian.
\newblock {\em J. Differential Geometry}, 1(1):43--69, 1967.

\bibitem{BerlineGetzlerVergne}
E.~Getzler N.~Berline and M.~Vergne.
\newblock {\em Heat Kernels and Dirac Operators}.
\newblock Springer Verlag, 1992.

\bibitem{PalaisSeminar}
R.S. Palais.
\newblock {\em Seminar on the Atiyah-Singer Index Theorem}, volume~57 of {\em
  Annals of Mathematics Studies}.
\newblock Princeton University Press, 1965.

\bibitem{RS}
M.~Raussen and C.~Skau.
\newblock Interview with {M}ichael {A}tiyah and {I}sadore {S}inger.
\newblock {\em Notices of the AMS}, February, 2005, 2005.

\bibitem{Shanahan}
P.~Shanahan.
\newblock {\em The Atiyah-Singer Index Theorem}.
\newblock Springer Verlag, 1978.

\end{thebibliography}

\end{document}